\documentclass[12pt]{amsart}
\usepackage[dvips]{graphicx}
\usepackage{amsmath,graphics,xypic}
\usepackage{color}
\usepackage{amsfonts,amssymb,hyperref}
\usepackage{enumitem}
\usepackage[all]{xy}
\theoremstyle{plain}
\newtheorem{theorem*}{Theorem}
\newtheorem*{lemma*}{Lemma}
\newtheorem{corollary*}{Corollary}
\newtheorem*{proposition*}{Proposition}
\newtheorem{conjecture*}{Conjecture}
\newtheorem{theorem}{Theorem}[section]
\newtheorem{lemma}[theorem]{Lemma}

\theoremstyle{remark}

\newtheorem*{claim}{Claim}

\theoremstyle{definition}

\newcounter{commentcounter}

\def\tautwo{\tau^{(2)}}
\def\nng{\NN(G)}

\def\gl{\mbox{GL}}

\def\nng{\NN(G)}

\def\eul{\operatorname{Eul}}

\def\gl{\mbox{GL}} \def\Q{\Bbb{Q}}
\def\HH{\mathcal{H}}

\def\op{\operatorname}
\def\det{\op{det}}
  \def\Z{\Bbb{Z}} \def\R{\Bbb{R}} \def\C{\Bbb{C}}
   \def\ll{\langle} \def\rr{\rangle}
  \def\g{\gamma}  \def\bp{\begin{pmatrix}}
\def\sm{\setminus} \def\ep{\end{pmatrix}} \def\bn{\begin{enumerate}} 
   \def\en{\end{enumerate}}
\def\ba{\begin{array}} \def\ea{\end{array}}  
  \def\s{\sigma}   \def\ti{\tilde}\def\wti{\widetilde}
\def\id{\mbox{id}}   
  
\def\be{\begin{equation}} \def\ee{\end{equation}} 
   
 \def\hom{\mbox{Hom}}

\def\co{\colon}

\def\ol{\overline}

\def \NN{\mathcal{N}}

\def\d{\dagger}

\newcommand{\version}[1] %marks the date of last editing and compilation
{\begin{center} Last edited on #1\\
    Last compiled on \today
  \end{center}
}

\begin{document}

\title{The $L^2$-Alexander torsion is symmetric}

\author{J\'er\^ome Dubois}
\address{Universit\'e Blaise Pascal - Laboratoire de Math\'ematiques UMR 6620 - CNRS\\
Campus des C\'ezeaux - B.P. 80026\\
63171 Aubi\`ere cedex\\
France}
  \email{jerome.dubois@math.univ-bpclermont.fr}
   
   \author{Stefan Friedl}
   \address{Fakult\"at f\"ur Mathematik\\ Universit\"at Regensburg\\   Germany}
   \email{sfriedl@gmail.com}
   
\author{Wolfgang L\"uck}
\address{Mathematisches Institut\\ Universit\"at Bonn\\
Endenicher Allee 60\\ 53115 Bonn\\ Germany}
\email{wolfgang.lueck@him.uni-bonn.de}

\date{\today}

\subjclass[2010]{Primary 57M27; Secondary 57Q10}

\keywords{$L^2$-Alexander torsion, duality, Thurston norm}

\begin{abstract}
We show that the $L^2$-Alexander torsion of a 3-manifold is symmetric. 
This can be viewed as a generalization of the symmetry of the Alexander polynomial of a knot.\end{abstract}

\maketitle

%=====================================
\section{Introduction}

An \emph{admissible triple} $(N,\phi,\g)$ consists of an irreducible, orientable, compact
3--manifold $N\ne S^1\times D^2$ with empty or toroidal boundary, a non-zero class $\phi
\in H^1(N;\Z)=\hom(\pi_1(N),\Z)$ and a homomorphism $\g\co \pi_1(N)\to G$ such that $\phi$
factors through $\g$.

In~\cite{DFL14a,DFL14b} we used the $L^2$--torsion (see e.g.~\cite{Lu02})
to associate to  an admissible triple $(N,\phi,\g)$ 
the $L^2$--Alexander torsion  $\tautwo(N,\phi,\g)$ which is  a function
\[
\tautwo(N,\phi,\g)\co \R_{>0}\to \R_{\geq 0}
\]
that is well-defined up to multiplication by a  function of the type $t\mapsto t^m$ for some $m\in \Z$.
We recall the definition in Section~\ref{section:def}.

The goal of this paper is to show that  the $L^2$-Alexander torsion is symmetric.
In order to state the symmetry result we need to recall that given a
3-manifold $N$ the \emph{Thurston norm}~\cite{Th86} of some $\phi\in
H^1(N;\Z)=\hom(\pi_1(N),\Z)$ is defined as
 \[
x_N(\phi):=\min \{ \chi_-(S)\, | \, S \subset N \mbox{ properly embedded surface dual to }\phi\}.
\]
Here, given a surface $S$ with connected components $S_1\cup\dots \cup S_k$, we define its
complexity as
\[
\chi_-(S)=\sum_{i=1}^k \max\{-\chi(S_i),0\}.
\]
Thurston~\cite{Th86} showed that $x_N$ is a, possibly degenerate, norm on $H^1(N;\Z)$. 
 We can now formulate the main result of this paper.

 \begin{theorem}\label{mainthm2}
   Let $(N,\phi,\g)$ be an admissible triple. Then  for any  representative  $\tau$ of $\tautwo(N,\phi,\g)$ 
  there exists  an $n\in \Z$ with
   $n\equiv x_N(\phi) \mbox{ mod } 2$ such that
   \[
   \tau(t^{-1})=t^n\cdot \tau(t)\mbox{ for any }t\in \R_{>0}.
   \]
 \end{theorem}

 It is worth looking at the case that $N=S^3\sm\nu K$ is the complement of a tubular
 neighborhood $\nu K$ of an oriented knot $K\subset S^3$.  We denote by $\phi_K\co
 \pi_1(S^3\sm \nu K)\to \Z$ the epimorphism sending the oriented meridian to 1. Let
 $\g\colon \pi_1(N)\to G$ be a homomorphism such that $\phi_K$ factors through $\g$. We
 define
\[
\tautwo(K,\g):=\tautwo(S^3\sm \nu K,\phi_K,\g).
\]
If we take $\g=\id$ to be the identity, then we showed  in~\cite{DFL14b} that 
\[
\tautwo(K,\id)= \Delta^{(2)}_K(t)\cdot \max\{1,t\},
\]
where $\Delta^{(2)}_K(t)\colon \R_{>0}\to \R_{\geq 0}$ denotes the $L^2$-Alexander
invariant which was first introduced by Li--Zhang~\cite{LZ06a,LZ06b,LZ08} and which was
also studied in~\cite{DW10,DW13,BA13a,BA13b}.

If we take $\g=\phi_K$, then we showed in~\cite{DFL14b} that the $L^2$-Alexander torsion
$\tautwo(K,\phi_K)$ is fully determined by the Alexander polynomial $\Delta_K(t)$ of $K$
and that in turn $\tautwo(K,\phi_K)$ almost determines the Alexander polynomial
$\Delta_K(t)$. In this sense the $L^2$-Alexander torsion can be viewed as a `twisted'
version of the Alexander polynomial, and at least morally it is related to the twisted
Alexander polynomial of Lin~\cite{Li01} and Wada~\cite{Wa94} and to the higher-order
Alexander polynomials of Cochran~\cite{Co04} and Harvey~\cite{Ha05}. We refer to~\cite{DFL14a} 
for more on the relationship and similarities between the various twisted
invariants.

If $K$ is a knot, then any Seifert surface is dual to $\phi_K$ and it immediately follows
that $x(\phi_K)\leq \max\{2\cdot\mbox{genus(K)}-1,0\}$. In fact an elementary argument
shows that for any \emph{non-trivial} knot we have the equality
$x(\phi_K)=2\cdot\mbox{genus(K)}-1$. It follows in particular that the Thurston
norm of $\phi_K$ is odd. We thus obtain the following corollary to Theorem~\ref{mainthm2}.

\begin{theorem}
  Let $K\subset S^3$ be an oriented non-trivial knot and let $\g\colon \pi_1(N)\to G$ be a
  homomorphism such that $\phi_K$ factors through $\g$. Then there exists an odd $n$ with
  \[
   \tautwo(K,\g)(t^{-1})=t^n\cdot \tautwo(K,\g)(t)\mbox{ for any }t\in \R_{>0}.
   \]
\end{theorem}

The proof of Theorem~\ref{mainthm2} has many similarities with the
proof of the main theorem in~\cite{FKK12} which in turn builds on the ideas of 
Turaev~\cite{Tu86,Tu90,Tu01,Tu02}. In an attempt to keep the proof as short as possible we will
on several occasions refer to~\cite{FKK12} and~\cite{Tu90} for definitions and results.

Our symmetry theorem can be viewed as a variation on the theme that the ordinary Alexander
polynomials of knots and links are symmetric, and that the twisted Reidemeister torsions
corresponding to unitary representations are symmetric. We refer 
to~\cite{Ki96,FKK12,HSW10} for proofs of such symmetry results.  \medskip

The paper is structured as follows. In Section~\ref{sec:Euler_structures} we recall the
notion of an Euler structure which is due to Turaev. We proceed in 
Section~\ref{section:l2torsion} with the definition of the $L^2$-torsion of a manifold
corresponding to a representation over a group ring.  We prove our main technical duality
theorem, which holds for manifolds of any dimension, in Section~\ref{sec:duality_for_torsion}.  
In Section~\ref{sec:twisted_torsion_of_3-manifolds} we relate the
$L^2$-torsion of a 3-manifold to a relative $L^2$-torsion. Finally, in Section~\ref{section:proofmainthm} 
we introduce the $L^2$-Alexander torsion of an admissible
triple and we prove Theorem~\ref{mainthm2}.

\subsection*{Conventions.} All manifolds are assumed to be smooth, connected, orientable
and compact.  All CW-complexes are assumed to be finite and connected.  If $G$ is a group
then we equip $\C[G]$ with the usual involution given by complex conjugation and by
$\ol{g}:=g^{-1}$ for $g\in G$.  We extend this involution to matrices over $\C[G]$ by
applying the involution to each entry.

Given a ring $R$ we will view all modules as left $R$-modules, unless we say explicitly
otherwise.  Furthermore, given a matrix $A \in M_{m,n}(R)$, by a slight abuse of notation,
we denote by $A \colon R^m \to R^n$ the $R$-homomorphism of left $R$-modules obtained by
right multiplication with $A$ and thinking of elements in $R^m$ as the only row in a
$(1,m)$-matrix.

\subsection*{Acknowledgments.}
 The second author gratefully acknowledges the
support provided by the SFB 1085 `Higher Invariants' at the University of Regensburg,
funded by the Deutsche Forschungsgemeinschaft DFG.
The paper is also financially supported by the Leibniz-Preis of the third author granted by the
 {DFG}.

%=====================================
\section{Euler structures}
\label{sec:Euler_structures}

In this section we recall the notion of an Euler structure of a pair of CW-complexes and
manifolds which is due to Turaev.  We refer to~\cite{Tu90,Tu01,FKK12} for full details.
Throughout this paper, given a space $X$, we denote by $\HH_1(X)$ the first integral
homology group viewed as a multiplicative group.

%=====================================
\subsection{Euler structures on CW-complexes}

Let $X$ be a finite CW-complex of dimension $m$ and let $Y$ be a proper subcomplex.
We denote by $p \colon \wti{X}\to X$ the universal covering of $X$ and we write
$\wti{Y}:=p^{-1}(Y)$.  An \emph{Euler lift $c$} is a set of cells in $\wti{X}$ such that
each cell of $X\sm Y$ is covered by precisely one of the cells in the Euler lift.

 Using the canonical left action of $\pi=\pi_1(X)$ on $\wti{X}$ we obtain a free and
 transitive action of $\pi$ on the set of cells of $\wti{X}\sm \wti{Y}$ lying over a fixed
 cell in $X\sm Y$.  If $c$ and $c'$ are two Euler lifts, then we can order the cells such
 that $c=\{c_{ij}\}$ and $c'=\{c_{ij}'\}$ and such that for each $i$ and $j$ the cells
 $c_{ij}$ and $c_{ij}'$ lie over the same $i$-cell in $X\sm Y$.  In particular there exist
 unique $g_{ij}\in \pi$ such that $c_{ij}'=g_{ij}\cdot c_{ij}$.  We now write
 $\HH=\HH_1(X)$ and we denote the projection map $\pi\to \HH$ by $\Psi$.  We define
\[
c'/c:= \prod\limits_{i=0}^m\prod\limits_{j}\Psi(g_{ij})^{(-1)^i}\in  \HH.
\]
We say that $c$ and $c'$ are \emph{equivalent} if $c'/c\in \HH$ is trivial.
An equivalence class of Euler lifts will be referred to as an \emph{Euler structure}.
We  denote by $\eul(X,Y)$ the set of Euler structures.
If $Y=\emptyset$ then we will also write $\eul(X)=\eul(X,Y)$.

Given $g\in \HH$ and $e\in \eul(X,Y)$ we define $g\cdot e\in \eul(X,Y)$ as follows: pick a
representative $c$ for $e$ and pick $\wti{g}\in \pi_1(X)$ which represents $g$, then act
on one $i$-cell of $c$ by $ g^{(-1)^i}$.  The resulting Euler lift represents an element
in $\eul(X,Y)$ which is independent of the choice of the cell. We denote by $ g\cdot e$
the Euler structure represented by this new Euler lift.  This defines a free and
transitive $\HH$-action on $\eul(X,Y)$, with $( g\cdot e)/e= g$.
 
If $(X',Y')$ is a cellular subdivision of $(X,Y)$, then there exists a canonical
$\HH_1(X)$-equivariant bijection
\[
\s \colon \eul(X,Y)\to \eul(X',Y')
\] 
which is defined as follows: Let $e\in \eul(X,Y)$
and pick an Euler lift for $(X,Y)$ which represents $e$. There exists a unique Euler lift
for $(X',Y')$ such that the cells in the Euler lift of $(X',Y')$ are contained in the
cells of the Euler lift of $(X,Y)$. We then denote by $\s(e)$ the Euler structure
represented by this Euler lift.  This map agrees with the map defined by Turaev~\cite[Section~1.2]{Tu90}.

%=====================================
\subsection{Euler structures of smooth manifolds}

Now we will quickly recall the definition of Euler structures on smooth manifolds.  Let
$N$ be a manifold and let $\partial_0N \subset \partial N$ be a union of components of
$\partial N$ such that $\chi(N,\partial_0 N)=0$.  We write $\HH=\HH_1(N)$.  A
\emph{triangulation} of $N$ is a pair $(X,t)$ where $X$ is a simplicial complex and
$t\colon |X|\to N$ is a homeomorphism. Note that $t^{-1}(\partial_0 N)$ is a simplicial
subspace of $X$.  Throughout this section we write $Y:=t^{-1}(\partial_0 N)$.  For the
most part we will suppress $t$ from the notation.  Following~\cite[Section~I.4.1]{Tu90} we
consider the projective system of sets $\{\eul(X,Y)\}_{(X,t)}$ where $(X,t)$ runs over all
$C^1$-triangulations of $N$ and where the maps are the $\HH$-equivariant bijections
between these sets induced either by $C^1$-subdivisions or by smooth isotopies in $N$.

Now we define $\eul(N,\partial_0 N)$ by identifying the sets $\{\eul(X,Y)\}_{(X,t)}$ via
these bijections. We refer to $\eul(N,\partial_0 N)$ as the set of Euler structures on
$(N,\partial_0 N)$.  Note that for a $C^1$-triangulation $X$ of $N$ we get a canonical
$\HH$-equivariant bijection $\eul(X,Y)\to \eul(N,\partial_0 N)$.

%%======================================
%\subsection{Smooth Euler structures}\label{section:smootheuler}
%Let $N$ be a compact manifold and let $\partial_0 N$ be a union of components of $\partial
%N$ such that $\chi(N,\partial_0 N)=0$.  Following Turaev (see~\cite[Section~5.1]{Tu90}) we
%define a \emph{regular vector field} on $(N,\partial_0 N)$ to be a nowhere vanishing
%vector field on $N$ which points inwards on $\partial_0 N$ and outwards on $\partial
%N\sm \partial_0 N$.  Two such vector fields are called \emph{homologous} if for some point
%$x\in \operatorname{Int}(N)$ the restrictions of the vector fields to $N\sm x$ are
%homotopic in the set of all regular vector fields on $N\sm x$.  The set of homology
%classes of regular vector fields is called $\vect(N,\partial_0 N)$.  By Turaev~\cite{Tu90}
%the set $\vect(N,\partial_0 N)$ admits a canonical, free and transitive action by
%$\HH_1(N)$, and there exists a canonical $\HH_1(N)$-equivariant bijection
%\[
%\ca_N\colon \eul(N,\partial_0 N)\to \vect(N,\partial_0 N).
%\]

% ==============================================================================
\section{The $L^2$-torsion of a manifold}\label{section:l2torsion}

% ==============================================================================
\subsection{The Fuglede-Kadison determinant and the $L^2$-torsion of a chain complex}

Before we start with the definition of the $L^2$-Alexander torsion we need to recall some
key properties of the Fuglede-Kadison determinant and the definition of the $L^2$-torsion
of a chain complex of free based left $\C[G]$-modules. Throughout the section we refer 
to~\cite{Lu02} and to~\cite{DFL14b} for details and proofs.

We fix a group $G$.  Let $A$ be a $k\times l$-matrix over $\C[G]$.  Then there exists the
notion of $A$ being of `determinant class'.  (To be slightly more precise, we view the
$k\times l$-matrix $A$ as a homomorphism $\NN(G)^l\to \NN(G)^k$, where $\NN(G)$ is the von
Neumann algebra of $G$, and then there is the notion of being of `determinant class'.)  We
treat this entirely as a black box, but we note that if $G$ is residually amenable, e.g., a
3-manifold group~\cite{He87} or solvable, then by~\cite{Lu94,Sc01,Cl99,ES05} any matrix
over $\Q[G]$ is of determinant class.   If the matrix $A$ is not of determinant class then
for the purpose of this paper we define $\det_{\nng}(A)=0$. On the other hand, if $A$ is
of determinant class, then we define
\[
\det_{\nng}(A):=\mbox{Fuglede-Kadison determinant of $A$}\in \R_{>0}.
\] 
Note that we do
not assume that $A$ is a square matrix.  We will not provide a definition of the
Fuglede-Kadison determinant but we summarize a few key properties in the following theorem
which is basically a consequence of~\cite[Example~3.12]{Lu02} and~\cite[Theorem~3.14]{Lu02}.

\begin{theorem}\label{thm:propfk}
  \bn
\item[$(1)$] If $A$ is a square matrix with complex entries such that the usual
  determinant $\det(A)\in \C$ is non-zero, then $\det_{\nng}(A)=|\det(A)|$.
  % \item[$(2)$] If $A$ is of determinant class, then $\det_{\nng}(A)>0$.
\item[$(2)$] The Fuglede-Kadison determinant does not change if we swap two rows or two
  columns.
\item[$(3)$] Right multiplication of a column by $\pm g$ with $g\in G$ does not change the
  Fuglede--Kadison determinant.
\item[$(4)$] For any matrix $A$ over $\C[G]$ we have
  \[
\det_{\nng}(A)=\det_{\nng}(\ol{A}^t).
\] 
\en
\end{theorem}

Note that (2) implies that when we study determinants of homomorphisms we can work with
unordered bases.

Now let
\[
C_*=\left( \,\,0\to C_l\xrightarrow{\partial_l} C_{l-1}\xrightarrow{\partial_{l-1}}
  \dots C_1 \xrightarrow{\partial_1} C_0\to 0\right)
\] 
be a chain complex of free left $\C[G]$-modules. We can then consider the corresponding $L^2$-Betti numbers
$b_i^{(2)}(C_*)\in \R_{\geq 0}$, as defined in~\cite{Lu02}. 

Now suppose that the chain complex is equipped with bases $B_i\subset C_i$, $i=0,\dots,l$.
If at least one of the $L^2$-Betti numbers $b_i^{(2)}(C_*)$ is non-zero or if at least one the boundary
maps is not of determinant class, then we define the $L^2$-torsion
$\tautwo(C_*,B_*):=0$. Otherwise we define the $L^2$-torsion of the based chain complex
$C_*$ to be
\[
\tautwo(C_*,B_*):=\prod_{i=1}^l \det_{\nng}(A_i)^{(-1)^i}\in \R_{>0}
\] 
where the $A_i$
denote the boundary matrices corresponding to the given bases.  Note that this definition
is the multiplicative inverse of the exponential of the $L^2$-torsion as defined in~\cite[Definition~3.29]{Lu02}.

%==============================================================================
\subsection{The  twisted $L^2$-torsion of a pair of CW-complexes}
Let $(X,Y)$ be a pair of finite CW-complexes and let $e\in \eul(X,Y)$.  We denote by
$p\colon \wti{X}\to X$ the universal covering of $X$ and we write
$\wti{Y}:=p^{-1}(Y)$. Note that the deck transformation turns $C_*(\wti{X},\wti{Y})$
naturally into a chain complex of left $\Z[\pi_1(X)]$-modules.

Now let $G$ be a group and let $\varphi\colon \pi(X)\to \gl(d,\C[G])$ be a
representation. We view elements of $\C[G]^d$ as row vectors. Right multiplication via
$\varphi(g)$ thus turns $\C[G]^d$ into a right $\Z[\pi_1(X)]$-module.  We then consider
the chain complex
\[
C_*^{\varphi}(X,Y;\C[G]^d):=\C[G]^d\otimes_{\Z[\pi_1(X)]}C_*(\wti{X},\wti{Y})
\] 
of left $\C[G]$-modules.

Now let $e\in \eul(X,Y)$.  We pick an Euler lift $\{{c}_{ij}\}$ which represents $e$.
Throughout this paper we denote by $v_1,\dots,v_d$ the standard basis for $\C[G]^d$.  We
equip the chain complex $ C_*^{\varphi}(X,Y;\C[G]^d)$ with the basis provided by the
$v_k\otimes {c}_{ij}$. Therefore we can define
\[
\tautwo(X,Y,\varphi,e):=\tautwo\left(C_*^{\varphi}(X,Y;\C[G]^d),\{v_k\otimes c_{ij}\}\right)\in \R_{\geq 0}.
\]

We summarize a few properties of the $L^2$-torsion in the following lemma.

\begin{lemma}\label{lem:proptautwo}
  \bn
\item[$(1)$] The number $\tautwo(X,Y,\varphi,e)$ is well--defined, i.e., independent of the
  choice of the Euler lift which represents $e$.
\item[$(2)$] If $g\in \HH_1(X)$, then
  \[
\tautwo(X,Y,\varphi,ge)=\det_{\nng}(\varphi(g^{-1}))\cdot \tautwo(X,Y,\varphi,e).
\]
\item[$(3)$] If $\delta\colon \pi_1(X)\to \gl(d,\C[G])$ is conjugate to $\varphi$, i.e., if
  there exists an $A\in \gl(d,\C[G])$ such that $\delta(g)=A\varphi(g)A^{-1}$ for all
  $g\in \pi_1(X)$, then
  \[\tautwo(X,Y,\delta,e)=\tautwo(X,Y,\varphi,e).\]
\item[$(4)$] If $(X',Y')$ is a cellular subdivision of $(X,Y)$ and if $e'\in \eul(X',Y')$
  is the Euler structure corresponding to $e$, then
  \[
\tautwo(X',Y',\varphi,e')=\tautwo(X,Y,\varphi,e).
\] 
\en
\end{lemma}

The proofs are completely analogous to the proofs for ordinary Reidemeister torsion as
given in~\cite{Tu86,Tu01,FKK12}.  In the interest of time and space we will therefore not
provide the proofs.

%==============================================================================
\subsection{The  $L^2$-Alexander torsion for manifolds}

Let $N$ be a manifold and let $\partial_0N \subset \partial N$ be a union of components of
$\partial N$.  Let $G$ be a group and let
$\varphi\colon \pi(N)\to \gl(d,\C[G])$ be a representation. Finally let $e\in
\eul(N,\partial_0 N)$.

Recall that for any $C^1$-triangulation $f\colon X \to N$ we get a canonical bijection
$\eul(X,Y)\xrightarrow{f_*} \eul(N,\partial_0 N)$.  Now we define
\[
\tautwo(N,\partial_0N,\varphi,e):=\tautwo(X,Y,\varphi\circ f_*,f_*^{-1}(e)).
\]
By Lemma~\ref{lem:proptautwo} (4) and the discussion in~\cite{Tu90} the invariant
$\tautwo(N,\partial_0N,\varphi,e)\in \R_{\geq 0}$ is well-defined, i.e., independent of the
choice of the triangulation.

%=========================================
\section{Duality for torsion of manifolds equipped with Euler structures}
\label{sec:duality_for_torsion}

%=====================================
\subsection{The algebraic duality theorem for $L^2$-torsion}

Let $G$ be a group and let $V$ be a right $\C[G]$-module. We denote by $\ol{V}$ the left
$\C[G]$-module with the same underlying abelian group together with the module structure
given by $v\cdot_{\ol{V}} p:=\ol{p}\cdot_{V} v$ for any $p\in \C[G]$ and $v\in V$.  If $V$
is a left $\C[G]$-module then we can consider $\hom_{\C[G]}(V,\C[G])$ the set of all left
$\C[G]$-module homomorphisms.  Note that the fact that the range $\C[G]$ is a
$\C[G]$-bimodule implies that $\hom_{\C[G]}(V,\C[G])$ is naturally a right $\C[G]$-module.

In the following let $C_*$ be a chain complex of length $m$ of left $\C[G]$-modules
with boundary operators $\partial_i$.
Suppose that $C_*$ is equipped with a  basis $B_i$ for each $C_i$.
We denote by $C^\#$ the \emph{dual chain complex}
whose  chain groups
are  the $\C[G]$-left modules $C^\#_i:=\ol{\hom_{\C[G]}(C_{m-i},\C[G])}$ and
%$\partial_i:C_{i+1}^\#\to C_i^\#$
where the boundary map $\partial^\#_i:C_{i+1}^\#\to C_i^\#$
is given by $(-1)^{m-i}\partial_{m-i-1}^*$. This means that for any 
$c\in C_{m-i}$ and $d\in C^\#_{i+1}$ we have $\partial^\#_i(d)(c) = (-1)^{m-i}d(\partial_{m-i-1}(c))$.
We denote by $B_*^\#$  the bases of $C^\#$ dual to the bases $B_*$.
We have the following lemma.

\begin{lemma}\label{lem:dual}
If $\tautwo(C_*,B_*)=0$, then $\tautwo(C_*^\#,B_*^\#)=0$, otherwise we have
\[\tautwo(C_*,B_*) =\tautwo(C_*^\#,B_*^\#)^{(-1)^{m+1}}.\]
\end{lemma}

\begin{proof}
  We first note that by the proof of~\cite[Theorem~1.35 (3)]{Lu02} the $L^2$-Betti numbers
  of $C_*$ vanish if and only if the $L^2$-Betti numbers of $C_*^\#$ vanish.  If either does not vanish, then it thus follows that the other does not vanish, and both torsions  are zero by
  definition.

  We now suppose that the $L^2$-Betti numbers of $C_*$ vanish.  We denote by $A_i$ the
  matrices of the boundary maps of $C_*$ with respect to the given basis. It follows
  easily from the definitions that the boundary matrices of the chain complex $C_*^\#$
  with respect to the basis $B_*^\#$ are given by $(-1)^{m-i}\ol{A}_i^\d$.  The lemma is
  now an immediate consequence of the definitions and of Theorem~\ref{thm:propfk} (2).
\end{proof}

%=====================================
\subsection{The duality theorem for manifolds}

Before we state our main technical duality theorem we need to introduce two more
definitions.  \bn
\item Let $G$ be a group and let $\varphi\colon \pi\to \gl(d,\C[G])$ be a
  representation. We denote by $\varphi^\d$ the representation which is given by $g\mapsto
  \ol{\varphi(g^{-1})}^t$.
\item Let $N$ be an $m$-manifold and let $e\in \eul(N,\partial N)$.
 We pick a triangulation $X$ for $N$. We denote by $Y$ the subcomplex corresponding to $\partial N$.  Let $X^\d$ be the CW-complex which is given by the
  cellular decomposition of $N$ dual to $X$. We pick an Euler lift $\{c_{ij}\}$ which represents
  $e\in \eul(X,Y)=\eul(N,\partial N)$.  For any $i$-cell $c$ in $\wti{X}$ we denote by $c^\d$ the
  unique oriented $(m-i)$-cell in $\wti{X^\d}$ which has intersection number $+1$ with
  $c_{ij}$.  The  Euler lift $\{c^\d_{ij}\}$ defines an element in $\eul(X^\d)=\eul(N)$ that we denote by $e^\d$.
   We refer to~\cite[Section~1.4]{Tu01} 
  and~\cite[Section~4]{FKK12} for details. 
  \en

In this section we will prove the following duality theorem.

\begin{theorem}\label{thm:dualitygeneral}
  Let $N$ be an $m$-manifold.  Let $G$ be a group and
  let $\varphi\colon \pi(N)\to \gl(d,\C[G])$ be a representation. Let $e\in
  \eul(N,\partial N)$. Then either both $\tautwo(N,\partial N,\varphi,e)$ and
  $\tautwo(N,\varphi^\d,e^\d)$ are zero,
     or the following equality holds:
  \[ 
  \tautwo(N,\partial N,\varphi,e)={\tautwo(N,\varphi^\d,e^\d)^{(-1)^{m+1}}}.
\]
\end{theorem}

\begin{proof}
 As above we pick a triangulation $X$ for $N$ and  we denote by $Y$ the subcomplex corresponding to $\partial N$.  Let $X^\d$ be the CW-complex which is given by the
  cellular decomposition of $N$ dual to $X$. In the following we make the identification
  $\pi=\pi_1(X)=\pi_1(N)=\pi_1(X^\d)$.

  For the remainder of this section we pick an Euler lift $\{c_{ij}\}$ which represents
  $e\in \eul(N,\partial N)=\eul(X,Y)$.  We denote by $c_{ij}^\d$ the corresponding dual cells.   Theorem~\ref{thm:dualitygeneral}  follows immediately from the definitions
  and the following claim.

  \begin{claim}
    Either both $\tautwo (C_*^\varphi(X,Y;\C[G]^d),\{v_k\otimes c_{ij}\})$ and
    $\tautwo(C_*^{\varphi^\d}(X^\d;\C[G]^d),\{v_k\otimes c_{ij}^\d\})$ are zero, or the
    following equality holds:
    \[
    \tautwo \big(C_*^\varphi(X,Y;\C[G]^d),\{v_k\otimes c_{ij}\}\big)
    ={\tautwo\big(C_*^{\varphi^\d}(X^\d;\C[G]^d),\{v_k\otimes  c_{ij}^\d\}\big)^{(-1)^{m+1}}}.
\]
  \end{claim}

In order to prove the claim we first note that  there is a unique, non-singular intersection $\Z$-linear pairing
\[
C_{m-i}(\wti{X},\wti{Y})  \times  C_{i}(\wti{X^\d})\to \Z
\]
with the property that $a\cdot b^\d=\delta_{ab}$ for any cell $a$ of $\wti{X}\sm \wti{Y}$
and any cell $b$ of $\wti{X}$. We then consider the following pairing:
\[
\ba{ccl} C_{m-i}(\wti{X},\wti{Y}) \times C_{i}(\wti{X^\d})    &\to &\Z[\pi]\\
(a,b) &\mapsto & \ll a,b\rr :=\sum_{g\in \pi} (a\cdot gb)g^{-1}.\ea 
\]
Note that this pairing is sesquilinear in the sense that for any $a\in
C_{m-i}(\wti{X},\wti{Y})$, $b\in C_{i}(\wti{X^\d})$ and $p,q\in \Z[\pi]$ we have $\ll
pa,qb\rr =q\ll a,b\rr \ol{p}$. It is furthermore straightforward to see that the pairing
is non-singular.  This pairing has the property (see e.g.~\cite[Claim~14.4]{Tu01}) that
the following diagram commutes:
\[
\ba{cclcl} C_{i+1}(\wti{X},\wti{Y}) &\times &C_{m-i-1}(\wti{X^\d})&\to &\Z[\pi]\\[2mm]
\downarrow \partial_i&&\hspace{0.7cm} \uparrow(-1)^{i+1}\partial_{m-i-1}&&\downarrow = \\[2mm]
 C_{i}(\wti{X},\wti{Y})   &\times & C_{m-i}(\wti{X^\d}) & \to &\Z[\pi].\ea 
\]
Put differently, the maps
\[
\ba{rcl} C_i(\wti{X},\wti{Y})&\to&
\ol{\hom_{\Z[\pi]}(C_{m-i}(\wti{X^\d}),\Z[\pi])}\\
a&\mapsto & (b\mapsto \ll a,b\rr)\ea 
\]
define an isomorphism of based chain complexes of right $\Z[\pi]$-modules.  In fact it
follows easily from the definitions that the maps define an isomorphism
\[
 (C_*(\wti{X},\wti{Y}),\{c_{ij}\})\to (\ol{\hom_{\Z[\pi]}(C_{m-*}(\wti{X^\d}),\Z[\pi])},\{(c_{ij}^\d)^*\})
\]
of based chain complexes of left $\Z[\pi]$-modules. 
Tensoring these chain complexes with
$\C[G]^d$ we obtain an isomorphism
\begin{multline*}
\left(\C[G]^d\otimes_{\Z[\pi]} C_*(\wti{X},\wti{Y}),\{v_k\otimes c_{ij}\}\right)
\\
\to \left(\C[G]^d\otimes_{\Z[\pi]} \ol{\hom_{\Z[\pi]}( C_{m-*}(\wti{X^\d})^{*},\Z[\pi])},\{(c_{ij}^\d)^*\otimes v_k\}\right)
\end{multline*}
of based chain complexes of $\C[G]$-modules. Furthermore the maps
\[
\ba{rcl} \C[G]^d\otimes_{\Z[\pi]}\ol{ \hom_{\Z[\pi]}(C_i(\wti{X^\d}),\Z[\pi])} &\to&  
\ol{\hom_{\C[G]}\left(C_i^{\varphi^\d}(X^\d;\C[G]^d),\C[G]\right)}\\
v\otimes f&\mapsto & \left(\ba{rcl}
 C_i^{\varphi^\d}(X^\d;\C[G]^d)&\to &\C[G]
 \\ w\otimes \s &\mapsto &v \varphi\big(\ol{f(\s)}\big)\ol{w}^t \ea\right)\ea 
\]
induce
 an isomorphism
\[
\left(C_*^\varphi(X,Y;\C[G]^d),\{v_k\otimes c_{ij}\}\right)\to
\left(C_{*}^{\varphi^\d}(X^\d;\C[G]^d)^\#,\{(v_k\otimes c_{ij}^\d)^\#\}\right)
\]
of based chain complexes of left $\C[G]$-modules.
The claim is  now  an immediate consequence of Lemma~\ref{lem:dual}.
\end{proof}

%=====================================
\section{Twisted $L^2$-torsion of 3-manifolds}
\label{sec:twisted_torsion_of_3-manifolds}

Now we are heading towards the proof of Theorem~\ref{mainthm2}.
Therefore we are turning towards the study of $L^2$-torsions of 3-manifolds.  In order to
turn Theorem~\ref{thm:dualitygeneral} into the desired symmetry result we will need to
relate the $L^2$-torsions of a 3-manifold $N$ and the relative $L^2$-torsions of the pair
$(N,\partial N)$.  Henceforth we will restrict ourselves to one-dimensional
representations since these are precisely the ones which we will need in the proof of
Theorem~\ref{mainthm2}.

%=====================================
\subsection{Canonical  structures on tori}

Let $T$ be a torus.  We equip $T$ with a CW-structure with one 0-cell $p$, two 1-cells $x$
and $y$ and one 2-cell $s$.  We write $\pi=\pi_1(T,p)$ and by a slight abuse of notation
we denote by $x$ and $y$ the elements in $\pi$ represented by $x$ and $y$.  We denote by
$\wti{T}$ the universal cover of $T$. Then there exist lifts of the cells such that the
chain complex $C_*(\wti{T})$ of left $\Z[\pi]$-modules with respect to the bases given by
these lifts is of the form \be \label{equ:toruscc} 0\to \Z[\pi]\xrightarrow{\bp y-1&
  1-x\ep} \Z[\pi]^2\xrightarrow{\bp 1-x\\ 1-y\ep}\Z[\pi]\to 0.\ee
   We refer to the corresponding Euler structure of
$T$ as the \emph{canonical Euler structure on $T$}.  This definition is identical to the
definition provided by Turaev~\cite[p.~10]{Tu02}.

Given a group $G$ we say that a representation $\varphi\co \pi\to \gl(1,\C[G])$ is
\emph{monomial} if for any $x\in \pi$ we have $\varphi(x)=zg$ for some $z\in \C$ and $g\in G$. We
now have the following lemma.

\begin{lemma}\label{lem:torus1}
  Let $T$ be the torus and let $\varphi\colon \pi(T)\to \gl(1,\C[G])$ be a monomial
  representation such that $b_*^{(2)}(T;\C[G]^d)=0$. Let $e$ be the canonical Euler
  structure on $T$. Then
  \[
  \tautwo(T,\varphi,e)=1.
  \]
\end{lemma}

\begin{proof}
  In~\cite{DFL14b} we used the canonical Euler structure (even though we did not call it
  that way) to compute $\tautwo(T,\varphi)=1$.
\end{proof}

%=====================================
\subsection{Chern classes on 3-manifolds with toroidal boundary}
Let $N$ be a compact, orientable 3-manifold with toroidal incompressible boundary and let $e\in \eul(N,\partial N)$.

Let $X$ be a triangulation for $N$.  We denote the subcomplexes corresponding to the
boundary components of $N$ by $S_1\cup \dots\cup S_b$.  We denote by $p \colon \wti{X}\to
X$ and $p_i \colon \wti{S_i}\to S_i, i=1,\dots,b$ the universal covering maps of $X$ and
$S_i,i=1,\dots,b$. For each $i$ we identify a component of $p^{-1}(S_i)$ with $\wti{S_i}$

 We pick an Euler lift $c$ which represents $e$. For each boundary
torus $S_i$ we pick an Euler lift $\wti{s_i}$ to $\wti{S_i}\subset p^{-1}(S_i)\subset \wti{X}$ which represents the canonical Euler structure.
  The set of cells
$\{\ti{s}_1,\dots,\ti{s}_b,c\}$ defines an Euler structure $K(e)$ for $N$, which only
depends on $e$.  Put differently, we just defined a map $K\colon \eul(N,\partial N)\to
\eul(N)$ which is easily seen to be $\HH_1(N)$-equivariant.

Given $e\in \eul(N)$ there exists a unique element $g\in \HH_1(N)$ such that $e=g\cdot
K(e^\d)$.  Following Turaev~\cite[p.~11]{Tu02} we define $c_1(e):=g\in H_1(N;\Z)$ and we
refer to $c_1(e)$ as the \emph{Chern class of $e$}.
% Note that
%\comments{I added this. It's admittedly slightly awkward and incorrect mix of multiplicative and additive notation.}
%\be \label{equ:c1square} c_1(ge)=2g+c_1(e) \mbox{ for any $g\in \HH_1(N)$.}\ee

%=====================================
\subsection{Torsions of 3-manifolds}\label{section:3mfdtorsions}

Let $\pi$ and $G$ be groups and let $\varphi\colon \pi\to \gl(1,\C[G])$ be a monomial
representation.  It follows from the multiplicativity of the Fuglede-Kadison determinant,
see~\cite[Theorem~3.14]{Lu02}, that given $g\in \pi$ the invariant
$\det_{\nng}(\varphi(g))$ only depends on the homology class of $g$.  Put differently,
$\det_{\nng}\circ \varphi\co \pi\to\R_{\geq 0}$ descends to a map $ \det_{\nng}\circ
\varphi\colon H_1(\pi;\Z)\to \R_{\geq 0}$. We can now formulate the following theorem.

\begin{theorem}\label{thm:torsionboundary}
  Let $N$ be a 3-manifold which is either closed or which has toroidal and incompressible
  boundary.  Let $G$ be a group and let $\varphi\colon \pi(N)\to \gl(1,\C[G])$ be a
  monomial representation. Suppose that $b_*^{(2)}(\partial N;\C[G])=0$. Then for any
  $e\in \eul(N,\partial N)$ we have
  \[
  \tautwo(N,\partial N,\varphi,e^{\d})= \det_{\nng}(\varphi(c_1(e)))\cdot
  \tautwo(N,\varphi,e).
  \]
\end{theorem}

\begin{proof}
The assumption that  $b_*^{(2)}(\partial N;\C[G])=0$
together with the proof of~\cite[Theorem~1.35~(2)]{Lu02} implies that 
$b_*^{(2)}(N;\C[G])=0$ if and only if 
$b_*^{(2)}(N,\partial N;\C[G])=0$. 
If both are non-zero, then both torsions $\tautwo(N,\partial N,\varphi,e^{\d}) $ and $ \tautwo(N,\varphi,e)$ are zero. 
For the remainder of this proof we now assume that 
$b_*^{(2)}(N;\C[G])=0$.

We pick a triangulation $X$ for $N$. As usual we denote by $Y$ the subcomplex
corresponding to $\partial N$.  Let $e\in \eul(N,\partial N)=\eul(X,Y)$. We pick an Euler
lift $c_*$ which represents $e^\d$.  We denote the components of $Y$ by $Y_1\cup\dots \cup
Y_b$ and we pick $\ti{s}^1_*,\dots, \ti{s}^b_*$ as in the previous section.
We write
$\ti{s}_*=\ti{s}^1_*\cup \dots \cup \ti{s}^b_*$.  We denote by $\{\ti{s}_*\cup c_* \}$ the
resulting Euler lift for $X$. Recall that this Euler lift represents $K(e)$.  We have the
following claim.

\begin{claim}
\[
\tautwo\big(C_*^\varphi(X,Y;\C[G]),\{ c_*\}\big)=  \tautwo\big(C_*^\varphi(X;\C[G]),\{\ti{s}_*\cup c_*\}\big).
\]
\end{claim}

In order to prove the claim we consider the following  short exact sequence of chain complexes
\[
0\to \bigoplus_{i=1}^b C_*^\varphi(Y_i;\C[G])\to C_*^\varphi(X;\C[G])\to C_*^\varphi(X,Y;\C[G])\to 0,
\]
with the  bases
\[
\{ s^{i}_* \}_{i=1,\dots,b},\, \{ \ti{s}_*\cup c_* \} \mbox{ and } \{c_* \}.
\]
Note that these bases are in fact compatible, in the sense that the middle basis is the
image of the left basis together with a lift of the right basis. By Lemma~\ref{lem:torus1}
we have $\tautwo(C_*^\varphi(Y_i;\C[G]),\{\wti{s}_*^i\})=1$ for $i=1,\dots,b$.  Now it follows
from the multiplicativity of torsion, see~\cite[Theorem~3.35]{Lu02}, that
\[
\tautwo\big(C_*^\varphi(X,Y;\C[G]),\{c_* \}\big)=  \tautwo\big(C_*^\varphi(X;\C[G]),\{c_*\cup \ti{s}_* \}\big).
\]
Here we used that the complexes are acyclic.
This concludes the proof of the claim.

Finally it follows from this claim, the  definitions
and Lemma~\ref{lem:proptautwo}
that
\[ 
\ba{rcl}
\tautwo(N,\partial N,\varphi,e^\d)&=&\tautwo\big(C_*^\varphi(X,Y;\C[G]),\{c_* \}\big)\\[2mm]
&=&  \tautwo\big(C_*^\varphi(X;\C[G]), \{\ti{s}_*\cup c_*\} \big)\\
&=&\tautwo(N,\varphi,K(e^\d))\\
&=&\tautwo(N,\varphi,c_1(e)^{-1}e)=\det_{\nng}(\varphi(c_1(e)))\cdot \tautwo(N,\varphi,e).\ea 
\]
\end{proof}

%=====================================
\section{The symmetry of the $L^2$-Alexander torsion}\label{section:proofmainthm}

%==============================================================================
\subsection{The  $L^2$-Alexander torsion for $3$-manifolds}
\label{section:def}

Let $(N,\phi,\g\co \pi_1(N)\to G)$ be an admissible triple  and let $e\in \eul(N)$. 
Given  $t\in \R_{>0}$
we   consider the  representation
\[
\ba{rcl} \g_t\colon \pi_1(N)&\to & \gl(1,\C[G]) \\
g&\mapsto & ( t^{\phi(g)}\g(g)).\ea 
\]
Then we  denote by $\tautwo(N,\phi,\g,e)$  the  function 
\[
\ba{rcl} \tautwo(N,\phi,\g,e)\colon \R_{>0}&\to& \R_{\geq 0}\\
t&\mapsto &\tautwo(N,\g_t,e).\ea
\]
Note that for a different Euler class $e'$ we have $e'=ge$ for some $g\in \HH_1(N)$ and 
it follows from Lemma~\ref{lem:proptautwo} that
 \[
\tautwo(N,\phi,\g,ge)(t)=t^{-\phi(g)}\tautwo(N,\phi,g,e)(t) \mbox{ for all }t\in \R_{>0}.
\]
Put differently, the functions $\tautwo(N,\phi,\g,e)$ and $\tautwo(N,\phi,\g,ge)$ are equivalent.
We denote by $\tautwo(N,\phi,\g)$ the equivalence class of the functions  $\tautwo(N,\phi,\g,e)$
and we refer to  $\tautwo(N,\phi,\g)$
as the \emph{$L^2$-Alexander torsion of $(N,\phi,\g)$}.

%=====================================
\subsection{Proof of Theorem \ref{mainthm2}}
\label{sec:proof}

For the reader's convenience we recall the statement of the proof of Theorem~\ref{mainthm2}.\\

\noindent \textbf{Theorem 1.1.}\emph{
   Let $(N,\phi,\g)$ be an admissible triple. Then  for any  representative  $\tau$ of $\tautwo(N,\phi,\g)$ 
  there exists  an $n\in \Z$ with
   $n\equiv x_N(\phi) \mbox{ mod } 2$ such that
   \[
   \tau(t^{-1})=t^n\cdot \tau(t)\mbox{ for any }t\in \R_{>0}.
   \]}

\begin{proof}
Let $e\in \eul(N)$. We write $\tau=\tautwo(N,\g,\phi,e)$.
Let $t\in \R_{>0}$. It  follows easily from the definitions that 
$(\g_t)^{\d}=\g_{t^{-1}}$.  
Using  Theorems~\ref{thm:dualitygeneral} and~\ref{thm:torsionboundary} we see
that the following equalities hold:
\[
\ba{rcl} \tau(t)=\tautwo(N,\g,\phi,e)= \tautwo(N,\g_t,e)
&=& \tautwo(N,\partial N,(\g_t)^{\d},e^{\d})\\
&=& \tautwo(N,\partial N,\g_{t^{-1}},e^{\d})\\
&=& \det_{\nng}(\g_{t^{-1}}(c_1(e)))\cdot \tautwo(N,\g_{t^{-1}},e)\\
&=& \det_{\nng}\left(t^{-\phi(c_1(e))}c_1(e)\right)\cdot \tautwo(N,\g_{t^{-1}},e)\\
&=& t^{-\phi(c_1(e))}\cdot \tautwo(N,\g_{t^{-1}},e)\\
&=& t^{-\phi(c_1(e))}\cdot \tau(t^{-1}).
\ea 
\]
Now it suffices  to prove the following claim:

\begin{claim}
For any $\phi\in H^1(N;\Z)$ we have 
\[
\phi(c_1(e))=x_N(\phi)  \mbox{ mod }2.
\]
\end{claim}

 Let $S$ be a Thurston norm minimizing surface which is dual to
$\phi$. Since $N$ is irreducible and since $N\ne S^1\times D^2$ we can arrange that $S$
has no disk components.  Therefore we have
\[
x_N(\phi)\equiv \chi_-(S) \equiv  b_0(\partial S)\, \mod\, 2\Z.
\]  
On the other hand,  by~\cite[Lemma~VI.1.2]{Tu02}  and~\cite[Section~XI.1]{Tu02} we have that 
 \[
b_0(\partial S) \equiv c_1(e)\cdot S \,\mod\, 2\Z
\]
where $c_1(e)\cdot S$ is the intersection number of $c_1(e)\in H_1(N)=\HH_1(N)$ with
$S$. Since $S$ is dual to $\phi$, we obtain that
\[
\phi(c_1(e))\equiv c_1(e)\cdot S\equiv  b_0(\partial S)\equiv \chi_-(S)\equiv x_N(\phi) \mbox{ mod } 2\Z.
\]
This concludes the proof of the claim.

\end{proof}

%============================================
\subsection{Extending the main result to real cohomology classes}
A \emph{real admissible triple} $(N,\phi,\g)$ consists of an irreducible, orientable, compact
3--manifold $N\ne S^1\times D^2$ with empty or toroidal boundary, a non-zero class $\phi
\in H^1(N;\R)=\hom(\pi_1(N),\R)$ and a homomorphism $\g\co \pi_1(N)\to G$ such that $\phi$
factors through $\g$. Verbatim the same definition  as in Section \ref{section:def} associates to $(N,\phi,e)$ a function $\tautwo(N,\phi,e)\co \R_{>0}\to \R_{\geq 0}$ that is well-defined up to multiplication by a function of the form $t\mapsto t^r$ for some $r\in \R$. Furthermore, verbatim the same argument as in the proof of Theorem \ref{mainthm2} gives us the following result.

\begin{theorem}
   Let $(N,\phi,\g)$ be a real admissible triple. Then  for any  representative  $\tau$ of $\tautwo(N,\phi,\g)$ 
  there exists  an $r\in \R$ such that
   \[
   \tau(t^{-1})=t^r\cdot \tau(t)\mbox{ for any }t\in \R_{>0}.
   \]
\end{theorem}

The only difference to Theorem \ref{mainthm2} is that for real cohomology classes $\phi\in H^1(N;\R)$ we can not relate the  exponent $r$ to the Thurston norm of $\phi$.

%\version{02.10.2014}


\begin{thebibliography}{10}
\bibitem[BA13a]{BA13a}
F. Ben Aribi, {\em The $L^2$-Alexander invariant detects the unknot}, C. R. Math. Acad. Sci. Paris 351 (2013), 215--219.

\bibitem[BA13b]{BA13b}
F. Ben Aribi, {\em The $L^2$-Alexander invariant detects the unknot}, Preprint (2013), arXiv:1311.7342

\bibitem[Cl99]{Cl99}
B. Clair, {\em Residual amenability and the approximation of $L^2$-invariants},
Michigan Math. J.  46 (1999), 331--346.

\bibitem[Co04]{Co04}
T. Cochran,  {\em Noncommutative knot theory}, Algebr. Geom.
Topol. {4} (2004), 347--398.

\bibitem[DFL14a]{DFL14a}
J. Dubois, S. Friedl and W. L\"uck,
{\em Three flavors of twisted knot invariants}, preprint  (2014)

\bibitem[DFL14b]{DFL14b}
J. Dubois, S. Friedl and W. L\"uck,
{\em The $L^2$--Alexander torsion of 3-manifolds}, preprint (2014)

\bibitem[DW10]{DW10}
J. Dubois and C. Wegner, {\em $L^2$-Alexander invariant for torus knots}, C. R. Math. Acad. Sci. Paris 348 (2010), no. 21-22, 1185-1189.
\bibitem[DW13]{DW13}
J. Dubois and C. Wegner, {\em Weighted $L^2$-invariants and applications to knot theory},  {to appear in Commun. Contemp. Math.} (2014).


\bibitem[ES05]{ES05}
G. Elek and E. Szab\'o, {\em
Hyperlinearity, essentially free actions and $L^2$-invariants},
Math. Ann. 332 (2005), 421--441.

\bibitem[FKK12]{FKK12}
S. Friedl, T. Kim and T. Kitayama,
{\em Poincar\'e duality and degrees of twisted Alexander polynomials},  Indiana Univ. Math. J. 61 (2012), 147--192.

\bibitem[Ha05]{Ha05}
S. Harvey, {\em Higher--order polynomial invariants of 3--manifolds giving lower bounds for the
Thurston norm}, Topology {44} (2005), 895--945.

\bibitem[He87]{He87} J. Hempel, {\em Residual
finiteness for $3$-manifolds}, Combinatorial group theory and topology (Alta, Utah, 1984),
379--396, Ann.\ of Math. Stud., 111, Princeton Univ. Press, Princeton, NJ (1987)


\bibitem[HSW10]{HSW10}
J. Hillman, D. Silver and S. Williams, {\em
On reciprocality of twisted Alexander invariants}, Alg. Geom. Top. 10 (2010),  2017--2026.

\bibitem[Ki96]{Ki96} T. Kitano, {\em Twisted Alexander
polynomials and Reidemeister torsion}, Pacific J. Math.  174 (1996),  431--442.

\bibitem[LZ06a]{LZ06a}
W. Li and  W. Zhang, {\em An $L^2$-Alexander invariant for knots},  Commun. Contemp. Math. 8
(2006), no. 2, 167--187.
\bibitem[LZ06b]{LZ06b}
W. Li and W. Zhang, {\em An $L^2$-Alexander-Conway invariant for knots and the volume conjecture},
Differential geometry and physics, 303--312, Nankai Tracts Math., 10, World Sci. Publ., Hackensack, NJ, 2006.

\bibitem[LZ08]{LZ08}
W. Li and W. Zhang, {\em Twisted $L^2$-Alexander-Conway invariants for knots}, Topology and physics, 236--259, Nankai Tracts Math. 12, World Sci. Publ., Hackensack, NJ, 2008.

\bibitem[Li01]{Li01} X. S. Lin, {\em Representations of knot groups and twisted
Alexander polynomials}, Acta Math. Sin. (Engl. Ser.)  17,  no. 3 (2001), 361--380.

\bibitem[L\"u94]{Lu94}
W. L\"uck, {\em Approximating $L^2$-invariants by their finite-dimensional analogues}, Geom. Funct. Anal.  4 (1994), 455--481.

\bibitem[L\"u02]{Lu02} W. L\"uck, {\em
 $L\sp 2$-invariants: theory and applications to geometry and $K$-theory}, Ergebnisse
der Mathematik und ihrer Grenzgebiete. 3. Folge. A Series of Modern Surveys in Mathematics, 44.
Springer-Verlag, Berlin, 2002.

\bibitem[Sc01]{Sc01}
T. Schick, {\em $L^2$-determinant class and approximation of $L^2$--Betti numbers},
Trans. Amer. Math. Soc.  353, Number 8 (2001), 3247--3265.

\bibitem[Th86]{Th86} W. P. Thurston, {\em A norm for the homology of 3--manifolds}, Mem.
Amer. Math. Soc. {339} (1986), 99--130.

\bibitem[Tu86]{Tu86} V. Turaev, {\em
Reidemeister torsion in knot theory},
 Russian Math. Surveys {41} (1986), no. 1, 119--182.

\bibitem[Tu90]{Tu90}
V.  Turaev, {\em  Euler structures, nonsingular vector fields, and Reidemeister-type
torsions}, Math. USSR-Izv. {34} (1990), 627-662.

\bibitem[Tu01]{Tu01} V. Turaev, {\em Introduction to combinatorial torsions}, Birkh\"auser, Basel, (2001)

\bibitem[Tu02]{Tu02}
V. Turaev, {\em Torsions of 3--manifolds}, Progress in
Mathematics, {208}. Birkh\"auser Verlag, Basel, 2002.

\bibitem[Wa94]{Wa94}
M. Wada, {\em Twisted Alexander polynomial for finitely presentable groups}, Topology 33, no. 2 (1994), 241--256.
\end{thebibliography}
\end{document}